\documentclass[a4paper,12pt]{amsart}

\usepackage[mathcal]{euscript}   
\usepackage{mathrsfs}

\usepackage[shortlabels]{enumitem}
\usepackage[matrix,arrow]{xy}
\usepackage{comment}
\usepackage{amssymb,enumitem,verbatim,stmaryrd,xcolor,microtype,graphicx,needspace}
\usepackage[T1]{fontenc}
\usepackage[utf8]{inputenc}
\usepackage[english]{babel} 
\usepackage[top=3.3cm,bottom=3.3cm,left=3.25cm,right=3.25cm]{geometry}
\usepackage[bookmarksdepth=2,linktoc=page,colorlinks,linkcolor={red!80!black},citecolor={red!80!black},urlcolor={blue!80!black}]{hyperref}

\usepackage{tikz}\usetikzlibrary{matrix,arrows,decorations.markings}
\usepackage{tikz-cd}
\usetikzlibrary{automata,positioning,decorations.pathreplacing}
\newcommand\cprime\textquotesingle 

\usepackage[figuresleft]{rotating}
\usepackage{changepage}

\usepackage{resizegather}
\usepackage{multicol}
\tikzset{->-/.style={decoration={markings,mark=at position #1 with {\color{black}\arrow{>}}},postaction={decorate,very thick}}}
\tikzstyle{vertex}=[circle, draw, inner sep=0pt, minimum size=6pt]

\pgfrealjobname{autformsp1}
\usepackage[mathcal]{euscript}                               
\usepackage{mathptmx}                                        
\usepackage[colorinlistoftodos,bordercolor=orange,backgroundcolor=orange!20,linecolor=orange,textsize=footnotesize]{todonotes} \makeatletter \providecommand \@dotsep{5} \def\listtodoname{List of Todos} \def\listoftodos{\@starttoc{tdo}\listtodoname} \makeatother 

\allowdisplaybreaks 

\usepackage{etoolbox}
\makeatletter
\patchcmd{\@startsection}{\@afterindenttrue}{\@afterindentfalse}{}{}             
\patchcmd{\part}{\bfseries}{\bfseries\LARGE}{}{}
\patchcmd{\section}{\scshape}{\bfseries}{}{}\renewcommand{\@secnumfont}{\bfseries} 
\patchcmd{\@settitle}{\uppercasenonmath\@title}{\large}{}{}
\patchcmd{\@setauthors}{\MakeUppercase}{}{}{}
\addto{\captionsenglish}{} 
\addto{\captionsenglish}{} 
\addto{\captionsenglish}{} 
\makeatother

\usepackage{fancyhdr}

\theoremstyle{plain}
\newtheorem{thm}{Theorem}[section]
\newtheorem{cor}[thm]{Corollary}
\newtheorem{lemma}[thm]{Lemma}
\newtheorem{prop}[thm]{Proposition}

\newtheorem*{thm*}{Theorem}
\newtheorem*{lem*}{Lemma}
\newtheorem{conj}[thm]{Conjecture}

\theoremstyle{definition}
\newtheorem{df}[thm]{Definition}
\newtheorem{rem}[thm]{Remark}

\newtheorem*{df*}{Definition}
\newtheorem*{ex*}{Example}
\newtheorem*{rem*}{Remark}

\DeclareRobustCommand{\gobblefour}[5]{}    

\DeclareFontFamily{OT1}{pzc}{}                                
\DeclareFontShape{OT1}{pzc}{m}{it}{<-> s * [1.10] pzcmi7t}{}
\DeclareMathAlphabet{\mathpzc}{OT1}{pzc}{m}{it}
\DeclareSymbolFont{sfoperators}{OT1}{bch}{m}{n} \DeclareSymbolFontAlphabet{\mathsf}{sfoperators} \makeatletter\def\operator@font{\mathgroup\symsfoperators}\makeatother 
\DeclareSymbolFont{cmletters}{OML}{cmm}{m}{it}
\DeclareSymbolFont{cmsymbols}{OMS}{cmsy}{m}{n}
\DeclareSymbolFont{cmlargesymbols}{OMX}{cmex}{m}{n}
\DeclareMathSymbol{\myjmath}{\mathord}{cmletters}{"7C}     \let\jmath\myjmath 
\DeclareMathSymbol{\myamalg}{\mathbin}{cmsymbols}{"71}     
\DeclareMathSymbol{\mycoprod}{\mathop}{cmlargesymbols}{"60}\let\coprod\mycoprod
\DeclareMathSymbol{\myalpha}{\mathord}{cmletters}{"0B}     \let\alpha\myalpha 
\DeclareMathSymbol{\mybeta}{\mathord}{cmletters}{"0C}      \let\beta\mybeta
\DeclareMathSymbol{\mygamma}{\mathord}{cmletters}{"0D}     \let\gamma\mygamma
\DeclareMathSymbol{\mydelta}{\mathord}{cmletters}{"0E}     \let\delta\mydelta
\DeclareMathSymbol{\myepsilon}{\mathord}{cmletters}{"0F}   \let\epsilon\myepsilon
\DeclareMathSymbol{\myzeta}{\mathord}{cmletters}{"10}      \let\zeta\myzeta
\DeclareMathSymbol{\myeta}{\mathord}{cmletters}{"11}       \let\eta\myeta
\DeclareMathSymbol{\mytheta}{\mathord}{cmletters}{"12}     \let\theta\mytheta
\DeclareMathSymbol{\myiota}{\mathord}{cmletters}{"13}      \let\iota\myiota
\DeclareMathSymbol{\mykappa}{\mathord}{cmletters}{"14}     \let\kappa\mykappa
\DeclareMathSymbol{\mylambda}{\mathord}{cmletters}{"15}    \let\lambda\mylambda
\DeclareMathSymbol{\mymu}{\mathord}{cmletters}{"16}        \let\mu\mymu
\DeclareMathSymbol{\mynu}{\mathord}{cmletters}{"17}        \let\nu\mynu
\DeclareMathSymbol{\myxi}{\mathord}{cmletters}{"18}        \let\xi\myxi
\DeclareMathSymbol{\mypi}{\mathord}{cmletters}{"19}        \let\pi\mypi
\DeclareMathSymbol{\myrho}{\mathord}{cmletters}{"1A}       \let\rho\myrho
\DeclareMathSymbol{\mysigma}{\mathord}{cmletters}{"1B}     \let\sigma\mysigma
\DeclareMathSymbol{\mytau}{\mathord}{cmletters}{"1C}       \let\tau\mytau
\DeclareMathSymbol{\myupsilon}{\mathord}{cmletters}{"1D}   \let\upsilon\myupsilon
\DeclareMathSymbol{\myphi}{\mathord}{cmletters}{"1E}       \let\phi\myphi
\DeclareMathSymbol{\mychi}{\mathord}{cmletters}{"1F}       \let\chi\mychi
\DeclareMathSymbol{\mypsi}{\mathord}{cmletters}{"20}       \let\psi\mypsi
\DeclareMathSymbol{\myomega}{\mathord}{cmletters}{"21}     \let\omega\myomega
\DeclareMathSymbol{\myvarepsilon}{\mathord}{cmletters}{"22}\let\varepsilon\myvarepsilon
\DeclareMathSymbol{\myvartheta}{\mathord}{cmletters}{"23}  \let\vartheta\myvartheta
\DeclareMathSymbol{\myvarpi}{\mathord}{cmletters}{"24}     \let\varpi\myvarpi
\DeclareMathSymbol{\myvarrho}{\mathord}{cmletters}{"25}    \let\varrho\myvarrho
\DeclareMathSymbol{\myvarsigma}{\mathord}{cmletters}{"26}  \let\varsigma\myvarsigma
\DeclareMathSymbol{\myvarphi}{\mathord}{cmletters}{"27}    \let\varphi\myvarphi

\DeclareMathOperator{\Hom}{Hom}

\DeclareMathOperator{\GL}{GL}
\DeclareMathOperator{\PGL}{PGL}

\DeclareMathOperator{\Bun}{{Bun}}
\DeclareMathOperator{\PBun}{{\mathbb{P}Bun}}
\DeclareMathOperator{\Pic}{{Pic}}
\DeclareMathOperator{\Ext}{{Ext}}

\newcommand\A{{\mathbb A}}

\newcommand\C{{\mathbb C}}

\newcommand\F{{\mathcal F}}
\newcommand\G{{\mathcal G}}

\renewcommand\P{{\mathbb P}}
\newcommand\Q{{\mathbb Q}}

\newcommand\cA{{\mathcal A}}

\newcommand\cH{{\mathcal H}}

\newcommand\cO{{\mathcal O}}

\newcommand{\E}{\mathcal E}
\newcommand{\Line}{\mathcal L}
\newcommand{\Fq}{\mathbb{F}_q}

\renewcommand\geq{\geqslant}
\renewcommand\leq{\leqslant}

\newcommand{\norm}[1]{|#1|}

\renewcommand\emptyset\varnothing

\title{On unramified automorphic forms over the projective line}

\author{Roberto Alvarenga}
\address{\rm Roberto Alvarenga, Instituto de Ci\^encias Matem\'aticas e de Computa\c{c}\~ao - USP, S\~ao Carlos, Brazil}
\email{alvarenga@icmc.usp.br}

\author{Valdir Pereira J\'unior}
\address{\rm Valdir Pereira J\'unior, Brazil}
\email{valdirjosepereirajunior@gmail.com}

\allowdisplaybreaks

\begin{document}

\begin{abstract} Let $q$ be a prime power and $\Fq$ be the finite field with $q$ elements. In this article we investigate the space of unramified automorphic forms for $\PGL_n$ over the rational function field defined over $\Fq$ (i.e.\ for $\mathbb{P}^1$ defined over $\Fq$). 
In particular, we  prove that the space of unramified cusp form is trivial and (for $n=3$) that the space of eigenforms is one dimensional. Moreover, we show that there are no nontrivial unramified toroidal forms for $\PGL_3$ over $\mathbb{P}^1$ and conjecture that the space of all toroidal automorphic forms is trivial.   
\end{abstract}

\maketitle


\section{Introduction}
\label{introduction} 


Automorphic forms are central objects in modern number theory, for instance, they play a key role in the celebrated Langlands program, see eg.\  \cite{langlands}. The goal of this work is to probe the space of unramified automorphic forms for $\PGL_n$ over the projective line.

Let $F$ be  either a local or a global field, such as the function field attached to a nonsingular projective curve. In the context of understanding the absolute Galois group $\mathrm{Gal}(\overline{F},F)$, which is a fundamental problem on number theory, automorphic forms play the following role. As the theory of representation of groups suggests, we may study $\mathrm{Gal}(\overline{F},F)$ through its finite dimension representations. 
Given a finite dimensional representation $\rho:\mathrm{Gal}(\overline{F},F) \rightarrow \GL_n(k) $, where $k$ equals to $ \C$ or $\mathbb{Q}_{\ell}$ for $\ell$ a prime number different from the characteristic of $F$, there exists $L(s, \rho)$ an analytic invariant attached to $\rho$, its $L$-function. Hence we might investigate $\rho$, and thus understand $\mathrm{Gal}(\overline{F},F)$, by means of its $L$-function. For $n=1$, this approach is fundamental for the development of abelian class field theory and the understanding of $\mathrm{Gal}(\overline{F},F)^{ab}$. 

On the other side, when $F$ is a global field, we might consider $\mathbb{A}$ its adelic ring and the so-called \textit{automorphic representations (forms)} $\pi$  of $\GL_n(\mathbb{A})$. Here is where emerge the main object of this paper. Attached to each automorphic representation $\pi$, we also have an analytic invariant $L(s, \pi)$, its $L$-function. The celebrated Langlands conjectures predict the existence of a correspondence between the $n$-dimensional representations of $\mathrm{Gal}(\overline{F},F)$  and the automorphic representation of $\GL_n(\mathbb{A})$ which preserves in some sense the respective $L$-functions. We refer \cite[Chap. 10 and 12]{dennis} for a complete discussion about how automorphic forms (representations) plays a role in the  Langlands program.

\subsection*{Cusp forms} When $F$ is the function field of a nonsingular projective geometrically irreducible curve $X$, automorphic forms with trivial central character and fixed ramification can be identified with continuous functions on the double quotient 
$$ \GL_n(F)Z(\mathbb{A}) \setminus \GL_n(\mathbb{A})/K$$ 
that are of moderate growth (see Section \ref{sec-back} for precise definitions).  Where $K$ stands for a compact open subgroup of $\GL_n(\mathbb{A})$ and $Z(\mathbb{A})$ is the center of $\GL_n(\mathbb{A})$.
Hence the domain of a \textit{cusp} automorphic form (see Definition \ref{def-cusp}) with fixed ramification is discrete modulo the right-action of a compact open subgroup $K$ of $\GL_n(\mathbb{A})$, which allows for explicit descriptions, see Lemma \ref{lemma-geocusp}.
In this context, roughly speaking,  the geometric Langlands correspondence(cf.\ \cite{drinfeld} and \cite{lafforgue})
states the existence of a bijective map between 
\begin{enumerate}

\item the set isomorphism classes of irreducible continuous representations 
\[\rho:\mathrm{Gal}(\overline{F},F) \rightarrow \GL_n(\overline{\Q_{\ell}}) \] which are unramified outside a finite number of places and whose determinant is of finite order and,

\item the space of $\overline{\Q}_{\ell}$-valued automorphic representations which are \textit{cuspidal} and whose central character is of finite order.
\end{enumerate}
Moreover, above correspondence is the bijection given by global class field theory for $n=1$ and ``preserves'' the respective $L$-functions for $n>1$. We prove in section \ref{sec-cusp} that the space of cusp forms over the projective line is trivial. See Definition \ref{def-cusp} for the definition of cusp forms.

\subsection*{$\mathcal{H}_{K}$-eigenforms} Let $\mathcal{O}_{\mathbb{A}} := \prod \mathcal{O}_x $ where the product is taken over all places $x$ of $F$ and  $\mathcal{O}_x $ is the ring of integers of the completion of $F$ with respect to $x$. When $K$ above is given by $\mathrm{GL}_n(\mathcal{O}_{\mathbb{A}})$, the standard maximal compact open subgroup of $\GL_n(\mathbb{A})$, we have the so-called \textit{unramified} automorphic forms.  For global function fields, the behave of unramified automorphic forms are particularly nice: every
unramified automorphic representation contains a unique unramified automorphic form,
or spherical vector, up to scalar multiples. Therefore the unramified automorphic representations correspond to certain $1$-dimensional representations of the (commutative)
spherical Hecke algebra $\mathcal{H}_K$, see Definition \ref{def-hecke}. In other words, unramified automorphic representations
correspond to eigenforms of $\mathcal{H}_K$ and are therefore determined by their eigenvalues
under the action of Hecke operators. We use the explicit description of the \textit{graphs of Hecke operators} introduced by Lorscheid in \cite{oliver-graphs} to parametrize the spaces of $\mathcal{H}_{K}$-eigenforms when $F$ is the rational function field.

\subsection*{Toroidal forms} Associated to a degree $n$ separable extension of $F$ there exists a maximal torus of $\GL_n$. The automorphic forms which vanish the integral along that torus, for every degree $n$ separable extension of $F$, are called \textit{toroidal} automorphic forms, cf.\ Definition \ref{def-toroidal}. When $F= \Q$, these special class of automorphic forms are closed related to the classical Riemann hypothesis, cf.\ \cite{zagier}. When $F$ is the rational function field over $\Fq$, we prove in section \ref{sec-toroidal}  that there does not exist nontrivial (unramified) toroidal automorphic forms. Moreover, we conjecture that this is the situation for the whole space of toroidal automorphic forms for every $n$.

\subsection*{Acknowledgements} The authors would like to thank Mikhail Kapranov for fruitful exchange of emails about formula (\ref{cuspeq}). The first author was supported by FAPESP 
[grant number 2022/09476-7].


\section{Background}\label{sec-back}

In this first section we set up the notation  used throughout the paper.
Let $F$ be a global function field defined over a finite field $\mathbb{F}_q$, where $q$ is a prime power. We might regard $F$ as the function field of a geometrically irreducible smooth projective curve $X$ defined over $\Fq$. We shall make it clear when we would like to specialize $X$ to be the projective line $\mathbb{P}^1$ defined over $\Fq$. The reason for employ both notations is because some of the definitions and properties which we write through the article is true for an arbitrary curve.

In the subsequent section we shall assume that $X$ is the projective line $\mathbb{P}^1$ defined over $\Fq$ and $F$ its function field. 

Let $g$ stand for the genus of $X$ and $|X|$ be the set of closed points of $X$ or, equivalently, the set of places in $F$. For $x \in |X|$, we denote by $F_x$ the completion of $F$ at $x$, by $\mathcal{O}_x$ the ring of integers of $F_x$, by $\pi_x \in \mathcal{O}_x$ (we can assume $\pi_x \in F$) a uniformizer of $x$ and by $q_x$ the cardinality of the residue field $\kappa(x):=\mathcal{O}_x/(\pi_x) \cong \mathbb{F}_{q_x}.$ Moreover, we denote by $|x|$  the degree of $x$, which is defined to be the extension fields degree $[\kappa(x) : \Fq]$. In other words, $q_x = q^{|x|}$. 
Let $| \cdot |_x$ the absolute value of $F_x$ (resp. $F$) such that $|\pi_x|_x = q_{x}^{-1},$ we call $| \cdot |_x$ the local norm for each $x \in |X.|$

Let $\mathbb{A}$ be the adele ring of $F$ and $\mathbb{A}^{*} $ the idele group. We denote $\mathcal{O}_{\mathbb{A}} := \prod \mathcal{O}_x $, where the product is
taken over all places $x$ of $F$. The idele norm is the quasi-character $|\cdot | : \mathbb{A}^{*} \rightarrow  \C^{*}$ that sends
an idele $(a_x) \in \mathbb{A}^{*}$ to the product $\prod |a_x |_x$ over all local norms. By the product formula, this defines a quasi-character on the idele class group $\mathbb{A}^{*}/ F^{*}.$

We might assume $F_x$ being embedded into the adele ring $\mathbb{A}$ by sending an element $a \in F_x$ to the adele $(a_y)_{y \in |X|}$ with $a_x = a$ and $a_y = 0$ for $y \neq x$. Not quite compatible with this embedding, we think of the unit group $F_{x}^{*}$ as a subgroup of the idele group $\mathbb{A}^{*}$  by sending an element $b$ of $F_{x}^{*}$
to the idele $(b_y)$ with $b_x = b$ and $b_y = 1$ for $y \neq x$. We will explain in case of ambiguity, which of these embeddings we use.

Let $G(\mathbb{A}):= \mathrm{GL}_n(\mathbb{A}),$ $Z(\mathbb{A})$ be the center of $G(\mathbb{A})$, $G(F):= \mathrm{GL}_n(F)$ and $K:= \mathrm{GL}_n(\mathcal{O}_{\mathbb{A}})$ the standard maximal compact open subgroup of $G(\mathbb{A})$. Note that $G(\mathbb{A})$ comes together with an adelic topology that turns $G(\mathbb{A})$ into a locally compact group. Hence $G(\mathbb{A})$ is endowed with a Haar measure. We fix the Haar measure on $G(\mathbb{A})$ for which $\mathrm{vol}(K)=1.$ The topology of $G(\mathbb{A})$ has a neighborhood basis $\mathcal{V}$ of the identity matrix that is given by all subgroups
$$K' = \prod_{x \in |X|} K_{x}' < \prod_{x \in |X|}K_x = K$$
where $K_x := \mathrm{GL}_n(\mathcal{O}_x)$, such that for all $x \in |X|$ the subgroup $K_{x}'$ of $K_x$ is open and consequently of finite index and such that $K_{x}^{'}$ differs from $K_x$ only for a finite number of places. 

\subsection*{Hecke algebra} Consider the space $C^0(G(\mathbb{A}))$ of continuous functions $f: G(\mathbb{A}) \rightarrow \C.$ Such a function is called \textit{smooth} if it is locally constant. The group $G(\mathbb{A})$ acts on $C^0(G(\mathbb{A}))$ through the right regular representation
$$\rho : G(\mathbb{A}) \rightarrow \mathrm{Aut}(C^0(G(\mathbb{A}))),$$ 
which is defined by right translation of the argument: $(g.f)(h):= (\rho(g)f)(h) := f(hg)$ for $g,h \in G(\mathbb{A})$ and $f \in C^0(G(\mathbb{A})).$

Let $H$ be a subgroup of $G(\mathbb{A})$. We say that $f \in C^{0}(G(\mathbb{A}))$ is \textit{left} or \textit{right} $H$-\textit{invariant} if for all $h \in H$ and $g \in G(\mathbb{A})$, 
$ f(hg) = f(g) \text{ or } f(gh)= f(g)$,  respectively.
If $f$ is right and left $H$-invariant, it is called \textit{bi}-$H$-\textit{invariant}.

\begin{df} \label{def-hecke} The complex vector space $\mathcal{H}$ of all smooth compactly supported functions $\Phi : G(\mathbb{A}) \rightarrow \C$ together with the convolution product
$$\Phi_1 \ast \Phi_2: g \longmapsto \int_{G(\mathbb{A})} \Phi_1(gh^{-1})\Phi_2(h)dh$$
for $\Phi_1, \Phi_2 \in \mathcal{H}$ is called the \textit{Hecke algebra for }$G(\mathbb{A})$. Its elements are called \textit{Hecke operators}.
\end{df}

The Hecke algebra $\mathcal{H}$ acts on $C^{0}(G(\mathbb{A}))$ by
$$\Phi(f): g \longmapsto \int_{G(\mathbb{A})} \Phi(h) f(gh) dh.$$
We say that a function $f \in C^{0}(G(\mathbb{A}))$ is $\mathcal{H}$-\textit{finite} if the space $\mathcal{H}\cdot f$ is finite dimensional.

The zero element of $\mathcal{H}$ is the zero function, but there is no multiplicative unit. For $K' \in \mathcal{V},$ we define $\mathcal{H}_{K'}$ to be the subalgebra of all bi-$K'$-invariant elements. These subalgebras have multiplicative units. Namely, the normalized characteristic function $\epsilon_{K'} := (\mathrm{vol}K')^{-1} \mathrm{char}_{K'}$ acts as the identity on $\mathcal{H}_{K'}$ by convolution. 

\begin{df}
When $K'=K$ above, we call $\mathcal{H}_K$ the \textit{spherical (unramified) part of} $\mathcal{H}$ and its elements are called \textit{spherical (unramified) Hecke operators}. 
\end{df}

Every $\Phi \in \mathcal{H}$ is bi-$K'$-invariant for some $K' \in \mathcal{V}$, see \cite[Lemma 1.4.3 and Prop. 1.4.4]{oliver-thesis}. In particular, 
$\mathcal{H} = \bigcup_{K' \in \mathcal{V}} \mathcal{H}_{K'}.$ Furthermore,
the spherical Hecke algebra $\mathcal{H}_K$ can be explicit described as follows.  We fix $ 1 \leq r \leq n$ an integer and let $x$ be a place of $F$. Let $\Phi_{x,r}$ stand for the characteristic function of
\[K\left( \begin{matrix}
\pi_x \rm{I}_r &  \\ 
 & \rm{I}_{n-r}
\end{matrix}   \right)K\] 
where $\rm{I}_r$ (resp. $ \rm{I}_{n-r}$) stands for the $r \times r$ (resp.  $(n-r) \times (n-r)$) identity matrix and the empty entry in the matrix means a zero entry. Observe that $\Phi_{x,n}$ is invertible and its inverse is given by the characteristic function of $K (\pi_x I_n)^{-1} K.$ The following theorem, due to Satake, describes $\mathcal{H}_K$ as a commutative (almost) polynomial algebra.

\begin{thm} \label{satake}  Identifying $\epsilon_K$ with $1 \in \C$ yields 
$$\mathcal{H}_K \cong \C[\Phi_{x,1}, \ldots, \Phi_{x,n}, \Phi_{x,n}^{-1}]_{x \in |X|}.$$
an isomorphism of $\C$-algebras. In particular, $\mathcal{H}_K $ is commutative. 
\end{thm}

\begin{proof} See \cite[Chapter 12, 1.6]{dennis}. 
\end{proof}

\subsection*{Automorphic forms}
A function $f \in C^{0}(G(\mathbb{A}))$ is called $K$-\textit{finite} if the complex vector space that is generated by $\{\rho(k)f\}_{k \in K}$ is finite dimensional.

We embed $G(\mathbb{A}) \hookrightarrow \mathbb{A}^{n^2 +1}$ via $g \mapsto (g, \det(g)^{-1}).$ We define a local height $\norm{ g_x}_x$ on $G(F_x) :=\mathrm{GL}_n(F_x)$ by 
restricting the height function 
$$(v_1, \ldots, v_{n^2+1}) \mapsto \mathrm{max}\{|v_1|_x, \ldots, |v_{n^2+1}|_x\}$$ 
on $F_{x}^{n^2+1}$. We note that $\norm{ g_x }_x \geq 1$ and that $\norm{ g_x}_x =1$ if $g_x \in K_x.$ We define the global height $\norm{ g}$ to be the product of the local heights. We say that $f \in C^{0}(G(\mathbb{A}))$ is of \textit{moderate growth} if there exists constants $C$ and $N$ such that
$$|f(g)|_{\C} \leq C \norm{ g}^{N}$$  
for all $g \in G(\mathbb{A}).$

Let $V \subset C^{0}(G(\mathbb{A}))$ be a sub-representation of $G(\A)$ on $C^{0}(G(\mathbb{A}))$. For $K' \in \mathcal{V}$, let $V^{K'}$ be the subspace of all $f \in V$ that are right $K'$-invariant i.e. such that $\rho(k')f=f$ for all $k'\in K'$. We say that the representation $V$ is \textit{admissible} if $V^{K'}$ is finite dimensional for every $K' \in \mathcal{V}$. In the following we take $V = \rho(G(\A))f$ for some 
$f \in C^{0}(G(\mathbb{A})).$ 

\begin{df} The \textit{space of automorphic forms} $\mathcal{A}$ (with trivial central character) is the complex vector space of all functions $f \in C^{0}(G(\mathbb{A}))$ which are smooth, $K$-finite, of moderate growth,  left $G(F)Z(\mathbb{A})$-invariant and such that the smooth representation $\rho(G(\A))f$ is admissible. Its elements are called \textit{automorphic forms}. 

\end{df}

\begin{rem} We are actually considering automorphic forms for $\PGL_n$. This is equivalent to consider in previous definition $G = \PGL_n$ and remove the left $Z(\A)$-invariance. However, for technical reasons, we maintain above definition and consider the Hecke algebra over $\GL_n.$ 
\end{rem}

\begin{lemma}  A function $f \in C^{0}(G(\mathbb{A}))$ is smooth and $K$-finite if and only if there is a $K' \in \mathcal{V}$ such that $f$ is right $K'$-invariant. In particular, 
$$V = \bigcup_{K' \in \mathcal{V}} V^{K'}$$ 
for every subspace $V \subseteq \mathcal{A}.$ 
\end{lemma}

\begin{proof} See \cite[Lemma 1.3.2]{oliver-thesis}.
\end{proof}

Hence, from  previous lemma, functions in $\mathcal{A}^{K'}$ can be identified with functions on the double quotient 
\[G(F)Z(\mathbb{A}) \setminus G(\mathbb{A})/K'\]
that are of moderate growth. We call $\cA^{K}$ by the space of \textit{unramified automorphic forms}.

\begin{rem}  Although the subspace $\cA^{K'}$ of $\cA$ is not stable under the action of $G(\A)$, it carries the induced action of the Hecke subalgebra $\cH_{K'}$. 
The Proposition \ref{propfund} re-writes the action of Hecke operators on automorphic forms given by previous integral as a finite sum. This is fundamental to define the \textit{graphs of Hecke operators} which shall be applied to parametrize the space of $\mathcal{H}_{K_x}$-eigenforms, see section \ref{sec-eigenforms}.\end{rem}

   
\section{Cusp forms}   \label{sec-cusp}

In this section we investigate the space of unramified cusp forms for $\PGL_n$ over the projective line. 

\begin{df} \label{def-cusp} An automorphic form $f \in \cA$ is a \textit{cusp form} (or cuspidal) if 
\[ \int_{U(F)\setminus U(\A)} f(ug)du = 0\]
for all $g \in G(\A)$ and all unipotent radicals subgroups $U$ of all standard parabolic subgroups $P$ of $G(\A)$. We denote the whole space of cusp forms by $\cA_{0}$.   If $f \in \cA^{K}$ is a cusp form we call it an \textit{unramified cusp form} and denote the whole space of unramified cusp forms by $\cA_{0}^{K}$. 
\end{df}

\subsection*{Geometric interpretation} Let $\Bun_nX$ be the set of
isomorphism classes of rank $n$ vector bundles on $X$. Let $\P\Bun_n X$ stands for the set of isomorphism classes of rank $n$ projective vector bundles. The vector bundles $\E, \E' \in \Bun_n X$ are in the same class in $\P\Bun_n X$ if there exists a line bundle $\Line \in \Pic X$ such that $\E \cong \E'\otimes \Line.$ In this case we denote $\overline{\E}=\overline{\E'}$ in $\P \Bun_n X.$ 

It is well known that the double quotient $\GL_1(F) \setminus \GL_1(\mathbb{A}) / \GL_1(\cO_{\A})$ is in bijection with the set of classes of divisors on $X$. Hence, it yields the following bijection
\[\GL_1(F) \setminus \GL_1(\mathbb{A}) / \GL_1(\cO_{\A}) \longleftrightarrow \Bun_1 X\]
since the set of classes of divisors on $X$ are in correspondence with the line bundles over $X$.
The following theorem, due to Weil, extends above bijection to higher rank vector bundles on $X$.

\begin{thm}[Weil] \label{weil-them} For every $n \geq 1$, there exists a bijection 
\begin{align*}
\GL_n(F) \setminus \GL_n(\mathbb{A}) / \GL_n(\cO_{\A}) & \longleftrightarrow  \Bun_n(X) \\
g & \longmapsto  \E_g
\end{align*} 
such that $\E_g \otimes \Line_a = \E_{ag}$ for $a \in \A^{\times}$ and $\Line_{a}$ the correspondent line bundle. Moreover, $\deg \E_g = \deg(\det g).$
\end{thm}

\begin{proof} Its well know that vector bundles are completely determined by its transition maps. The bijection follows, essentially, by associate to a vector bundle the adelic matrix given by the stalks of its transition maps. See either \cite[Lemma 5.1.6]{oliver-thesis} or \cite[Lemma 3.1]{frenkel} for a complete proof. 
\end{proof}

\begin{cor}  For every $n \geq 1$, there exists a bijection 
\[ \GL_n(F) Z(\mathbb{A}) \setminus \GL_n(\mathbb{A}) / \GL_n(\cO_{\A})  \longleftrightarrow  \P\Bun_n(X) \]
where $Z(\mathbb{A})$ is the center of $ GL_n(\mathbb{A})$
\end{cor}

\begin{rem} \label{rem-formsinbundles}
Therefore, due to Weil's theorem, we can interpret unramified automorphic forms $\cA^{K}$ as the space of complex valued functions on $\P\Bun_n(X)$ with some moderate growth condition. 
\end{rem}

The next lemma reinterprets the cuspidal condition in geometric terms. Despite it is already known and largely used in the literature, 
see for example \cite[Sec. 2]{kapranov} and \cite[pag. 296]{bump}, we could not find a proof. Thus we sketch its proof in the following for sake of completeness. 
We first observe, following \cite[Lemma 3]{harish-chandra}, that  $f \in \mathcal{A}^{K}$ is cuspidal if 
\[ \int_{U(F)\setminus U(\A)} f(ug)du = 0\]
for all $g \in G(\A)$ and all unipotent radicals subgroups $U$ of all \textit{maximal} parabolic subgroups $P$ of $G(\A)$.  

We fix $P$ a maximal parabolic subgroup of $G$. It is well known that there exists integers $r,s>0$ with
$r+s=n$ such that $P= UM$ where $U$ is the unipotent of $P$ and radical $M \cong \GL_r \times \GL_s$ is its Levi subgroup,
cf.\ \cite[Cor. 14.19]{Borel-A}. One says that $P$ is a parabolic maximal subgroup of type $(r,s)$. 
Observe furthermore that an element of $U$ has the following form 
\[\left( \begin{matrix}
\rm{I}_{r \times r} & h_{s \times r} \\
 0 & \rm{I}_{s \times s}
\end{matrix} \right) \]
where $\rm{I}_{r \times r}$ (resp. $\rm{I}_{s \times s}$) stands for the $r \times r $  (resp. $s \times s$) identity matrix and $h_{s \times r}$ is a matrix with $r$ rows and $s$ columns. 
That is, $U  \cong M_{r,s}$ where $M_{r,s}$ stands for the additive group of matrices with $r$ rows and $s$ columns. 

From Iwasawa decomposition, $G(\A) = P(\A)K$ where $P$ is the above fixed maximal parabolic subgroup. Given $g \in G(\A)$, we write $g = xmk$ where $x \in U(\A), m \in M(\A)$ and $k \in K.$ 
Hence,  $f \in \mathcal{A}^{K}$ is an unramified cusp form if 
\[ \int_{U(F)\setminus U(\A)} f(ug)du = \int_{U(F)\setminus U(\A)} f(uxmk)du =   
\int_{U(F)\setminus U(\A)} f(ym)dy = 0.\]
for all $m \in M(\A)$ and where $y = ux \in U.$ We write 
\[ y = \left( \begin{matrix}
\rm{I}_{r \times r} & h_{s \times r} \\
 0 & \rm{I}_{s \times s}
\end{matrix} \right)  
\quad \text{ and  } \quad 
m = \left( \begin{matrix}
h_r & 0 \\
 0 & h_s
\end{matrix} \right) \]
where $h_r \in \GL_r (\A)$, $h_s \in \GL_s (\A)$ and $h_{s \times r} \in M_{r,s}(\A)$. 
Let $\F \in \Bun_r X$ (resp. $\G \in \Bun_s X$) be the vector bundle which corresponds to $h_r$ (resp. $h_s$) 
via Theorem \ref{weil-them}.  
Applying once again Weil's theorem, $ym \in G(\A)$ corresponds to an extension of $\G$ by $\F$, see \cite[Sec. 7.3]{lepotier}
and \cite[Ex. 6.6]{shafarevich-2}.
Therefore, considering an unramified automorphic form as a complex valuated map from $\P\Bun_n X$ as observed in \ref{rem-formsinbundles}, above discussion implies the following lemma.

\begin{lemma} \label{lemma-geocusp}
An unramified automorphic form $f \in \cA^{K}$ is a cusp form if for any integers $r , s > 0$ with $r + s = n$ and any vector bundles $\F \in \Bun_{r}X$, $\G \in \Bun_{s}X$, 
\begin{equation}\label{cuspeq}
 \sum_{\E \in \Ext(\F,\G)} f(\E) =0,
\end{equation}
where we abuse the notation and write $\E$ meant the middle term of the correspondent exact sequence. 

\end{lemma}

\begin{thm} \label{thm-cusp} Let $F$ be the field of rational functions over $\Fq$ i.e.\ the function field of $\P^1$ the projective line defined over $\Fq$. Then $\cA_{0}^{K}$ is trivial for every $n \geq 2$. 
\end{thm}

\begin{proof}  We denote  the structural sheaf of $\P^1$ simply by $\cO$, i.e.\ $\cO := \cO_{\P^1}$. 
Hence the canonical sheaf $\omega := \omega_{\P^1}$ of $\P^1$ is isomorphic to $\cO(-2).$ 

Let $\F \in \Bun_r X$ and $\G \in \Bun_s X$, for some integers $r,s>0$. The Serre duality (see \cite[Sec. III.7]{biblia}) yields
\[\Ext^1(\F,\G)^{\vee} \cong \Hom(\G, \F \otimes \cO(-2)) = \mathrm{H}^0(\P^1, \F \otimes \cO(-2) \otimes \G^{\vee}).\]

From the Grothendieck classification (actually this was already known by Dedekind and Weber, see \cite{dedekind-weber})  of vector bundles on $\P^1$, see \cite[Thm. 11.51]{Gortz}, every rank $n$ vector bundle on the projective line is isomorphic to
$\mathcal{O}(d_1) \oplus \cdots \oplus \mathcal{O}(d_n)$
for some integers $ d_1 \leq \cdots \leq d_n.$ 
Hence we might write $\F := \mathcal{O}(k_1) \oplus \cdots \oplus \mathcal{O}(k_r)$ and 
$\G := \mathcal{O}(\ell_1) \oplus \cdots \oplus \mathcal{O}(\ell_s)$, for some integers 
$k_1 \leq \cdots \leq k_r$ and $\ell_1 \leq \cdots \leq \ell_s$.  Thus
\[ \F \otimes \omega \otimes \G^{\vee} = \bigoplus_{i,j} \cO(k_i - \ell_j -2).\]
Let $k := \max\{k_1, \ldots, k_r\}$ and $\ell := \min\{\ell_1, \ldots, \ell_s\}$. If $d - \ell < 2$, then
$\Ext^1(\F,\G)=\{0\}$
since $\Ext^1$ commutes with direct sum and $\Hom(\Line, \Line') = \{0\}$ if $\deg(\Line') > \deg(\Line)$ for $\Line, \Line'$ line bundles.  

Let $\E := \mathcal{O}(d_1) \oplus \cdots \oplus \mathcal{O}(d_n)$ be any rank n vector bundle on $\P^1$ with $d_1 \leq \cdots \leq d_n.$ If we take above  $\F = \cO(d_1)$ and $\G = \mathcal{O}(d_2) \oplus \cdots \oplus \mathcal{O}(d_n)$, thus $\Ext^1(\F,\G)=\{0\}$, i.e.\ $\Ext^1(\F,\G)$ consists only by the extension given by $\E$. 
Therefore, for $f \in \cA_{0}^{K}$ a cusp form, follows from above discussion and the geometric interpretation of cuspidal condition, Lemma \ref{lemma-geocusp}, that 
\[f(\overline{\E})=0,  \quad  \text{ for all } \E \in \Bun_n \P^1\]
and hence $\cA_{0}^{K}$ is trivial. 
\end{proof}

\begin{cor} There are also no Eisenstein series other than those induced from the Borel subgroup. In conclusion, $\mathcal{A}^K$ consists only of Eisenstein series on the Borel subgroup. \end{cor}

\begin{proof} Follows from previous theorem and \cite[Thm. 1.7.4]{valdir-thesis}.
\end{proof}

   
\section{The $\Phi_{x,r}$-eigenforms}  \label{sec-eigenforms} 

The aim of this section is to parametrize the space of unramified automorphic \linebreak
$\Phi_{x,r}$-eigenforms ($r=1,2$) for $\PGL_3$ over the rational function field. In order to achieve this goal, we need the graphs of Hecke operators introduced by Lorscheid in \cite{oliver-graphs}, see also \cite{oliver-toroidal} for applications of these graphs on the theory of automorphic forms.

\begin{df} We call $f \in \mathcal{A}$ a $\mathcal{H}_K$-eigenform with eigencharacter $\lambda_f$ if $f$ is an eigenvector for every $\Phi \in \mathcal{H}_K$ with eigenvalue $\lambda_f(\Phi).$ 
\end{df}

Note that $\lambda_f$ in the above definition defines a homomorphism of $\C$-algebras from $\mathcal{H}_K$ to $\C$. Hence $\lambda_f$ indeed defines an additive character on $\mathcal{H}_K.$

\begin{df} Let $x \in |X|$ be a closed point and $\underline{\lambda} := (\lambda_1, \ldots, \lambda_{n-1}) \in \C^{n-1}$. The space of $\Phi_{x,r}$-eigenforms (or $\mathcal{H}_{K_x}$-eigenforms), for $r=1, \ldots, n-1$, with eigenvalues $\underline{\lambda}$ is 
\[ \mathcal{A}(x, \underline{\lambda}) := \big\{ f \in \mathcal{A}^K \;\big|\; 
\Phi_{x,i}(f) = \lambda_i f \; \text{ for $i=1, \ldots, n-1$} \big\}
\]
where $\mathcal{A}^K$ is the space of unramified automorphic forms. 
\end{df}

\begin{rem} If $f \in \mathcal{A}^K$ is a $\mathcal{H}_K$-eigenform with eigencharacter $\lambda_f$, then $f$ is a $\mathcal{H}_{K_x}$-eigenform and $\lambda_r$ in the above definition is given by $\lambda_f (\Phi_{x,r})$ for $r=1, \ldots, n-1.$
\end{rem}

The goal of this section is to parametrize, for $n=3$, the space $\mathcal{A}(x, \lambda_1, \lambda_2)$ of $\Phi_{x,r}$-eigenforms ($r=1,2$) for some $x \in |\P^1|$ of degree one. As we said, we shall need to introduce the graphs of Hecke operators, which are graphs that encodes the action of Hecke operators on the space of automorphic forms.  For that reason, we need to
recall the following well-known proposition.  

\begin{prop} \label{propfund} Let $K' \in \mathcal{V}$ and fix $\Phi \in \mathcal{H}_{K'}$. For all  classes of adelic matrices $[g] \in G(F)Z(\A) \setminus G(\mathbb{A}) / K',$ there is a unique set of pairwise distinct classes $[g_1], \ldots,[g_r] \in G(F) Z(\A)\setminus G(\mathbb{A}) / K'$ and numbers $m_1,  \ldots, m_r \in \C^{*}$ such that  
$$\Phi(f)(g) = \sum_{i=1}^{r} m_i f(g_i).$$
for all $f \in \mathcal{A}^{K'}.$
\end{prop}

\begin{proof} See \cite[Prop. 1.6]{roberto-graphs}.
\end{proof}

\begin{df} Let $K' \in \mathcal{V}$ and fix $\Phi \in \mathcal{H}_{K'}$. For classes of adelic matrics  $[g],[g_1], \ldots,[g_r]$ in the double coset $ G(F)Z(\A) \setminus G(\mathbb{A}) / K'$ as in the last proposition, we denote  
$$\mathcal{V}_{\Phi,K'}([g]) := \big\{([g],[g_i],m_i)\big\}_{i=1, \ldots, r}.$$
We define the graph $\overline{\mathscr{G}}_{\Phi,K'}$ of the Hecke operator $\Phi$ (relative to $K'$) whose vertices are 
$$\mathrm{Vert} \; \overline{\mathscr{G}}_{\Phi,K'} = G(F)Z(\A) \setminus G(\mathbb{A}) / K'$$
and the oriented weighted edges 
$$\mathrm{Edge} \; \overline{\mathscr{G}}_{\Phi,K'} = \bigcup_{[g] \in \mathrm{Vert} \overline{\mathscr{G}}_{\Phi,K'}} \mathcal{V}_{\Phi,K'}([g]).$$
The classes $[g_i]$ are called the $\Phi$-neighbors of $[g]$ (relative to $K'$). 
\end{df}

\begin{rem} For $\Phi_{x,r} \in \mathcal{H}_K$ we adopt the shorthand notation 
$\overline{\mathscr{G}}_{x,r}^{(n)}$ for the graph $\overline{\mathscr{G}}_{\Phi_{x,r},K}$ and 
$\mathcal{V}_{x,r}([g])$ for the $\Phi_{x,r}$-neighborhood $\mathcal{V}_{\Phi_{x,r}, K}([g])$ of $[g]$, where $x \in |X|$ and $r=1, \ldots, n.$
\end{rem}

As we have seen in section \ref{sec-cusp}, we can see $f \in \mathcal{A}^K$ as a function on $\P\Bun_n X$. Hence we can give an algebraic geometry description of the graphs $\overline{\mathscr{G}}_{x,r}^{(n)}$ as follows.

The Weil Theorem \ref{weil-them} identifies the set of vertices of $\overline{\mathscr{G}}_{x,r}^{(n)}$  with the geometric objects in  $\PBun_nX$. Next we describe the edges of  $\overline{\mathscr{G}}_{x,r}^{(n)}$ in geometric terms. We say that two exact sequences of coherent sheaves on $X$
$$0 \longrightarrow \F_1 \longrightarrow \F \longrightarrow \F_{2} \longrightarrow 0 \hspace{0.2cm}\text{ and } \hspace{0.2cm} 0 \longrightarrow \F_{1}' \longrightarrow \F \longrightarrow \F_2' \longrightarrow 0$$  
are \textit{isomorphic with fixed $\F$} if there are isomorphism $\F_1 \rightarrow \F_1'$ and $\F_2 \rightarrow \F_2'$ such that 
\[\xymatrix@R5pt@C7pt{ 0 \ar[rr] && \F_1 \ar[rr] \ar[dd]^{\cong} && \F \ar[rr] \ar@{=}[dd] && \F_2 \ar[rr] \ar[dd]^{\cong} && 0 \\
&& && && && \\
0 \ar[rr] && \F_1' \ar[rr] && \F \ar[rr] && \F_2' \ar[rr] && 0 } \]
commutes. 
Let $\mathcal{K}_{x}$ be the torsion sheaf that is supported at $x$ and has stalk $\kappa(x)$ at $x,$ i.e. the skyscraper torsion sheaf at $x$. Fix $\E \in \mathrm{Bun}_n X$. For $r \in \{1, \ldots,n\},$ and $\E' \in \mathrm{Bun}_n X$ we define $m_{x,r}(\overline{\E},\overline{\E'})$ as the number of isomorphism classes of exact sequences
\[0 \longrightarrow \E'' \longrightarrow \E \longrightarrow \mathcal{K}_{x}^{\oplus r} \longrightarrow 0\]
with fixed $\overline{\E} \in \P\Bun_n X$ and where $\overline{\E''} \cong \overline{\E'}$ in $\P\Bun_n X.$ Similarly we can consider $m_{x,r}(\E,\E')$ by considering the isomorphisms in $\Bun_n X$ instead $\P\Bun_n X.$

\begin{df} \label{defconection} Let $x \in |X|$. For a projective vector bundle $\overline{\E} \in \P\mathrm{Bun}_n X$ we define
\[\mathcal{V}_{x,r}(\overline{\E}) := \big\{(\overline{\E}, \overline{\E'}, m) | m = m_{x,r}(\overline{\E},\overline{\E'}) \neq 0\big\},\]
and we call $\overline{\E'}$ a $\Phi_{x,r}$-neighbor of $\overline{\E}$ if $m_{x,r}(\overline{\E},\overline{\E'}) \neq 0$. In this case, $m_{x,r}(\overline{\E},\overline{\E'})$ is the multiplicity of $\overline{\E'}$ as a $\Phi_{x,r}$-neighbor of $\overline{\E}$.
\end{df}

The geometric interpretation of $\overline{\mathscr{G}}_{x,r}^{(n)}$  the graph of the Hecke operator $\Phi_{x,r} \in \mathcal{H}_K$ is given by the theorem below. 

\begin{thm} \label{theoremcorrespondence} Let $x \in |X|.$ The graph 
$\overline{\mathscr{G}}_{x,r}^{(n)}$ of $\Phi_{x,r}$ is described in geometric terms as 
$$\mathrm{Vert}\; \overline{\mathscr{G}}_{x,r}^{(n)} = \P\mathrm{Bun}_n X \hspace{0.3cm}\text{ and } \hspace{0.3cm} \mathrm{Edge} \;\overline{\mathscr{G}}_{x,r}^{(n)} = \coprod_{\E \in \P\mathrm{Bun}_n X} \mathcal{V}_{x,r}(\E).$$ 
\end{thm}

\begin{proof} See \cite[Thm. 3.4]{roberto-graphs}. 
\end{proof}

\begin{rem} The entire above discussion holds if we consider automorphic forms as functions on $G(F)\setminus G(\mathbb{A}) / K$ instead $Z(\A)G(F)\setminus G(\mathbb{A}) / K$ i.e.\ removing the action of $Z(\A)$. In this case we should replace $\P \Bun_n X$ by $\Bun_n X$ and denote $\overline{\mathscr{G}}_{\Phi,K'}$ (resp. $\overline{\mathscr{G}}_{x,r}^{(n)}$) simply by $\mathscr{G}_{\Phi,K'}$ (resp. $\mathscr{G}_{x,r}^{(n)}$). We refer \cite[Sec. 3]{roberto-graphs} and \cite[Sec. 1]{rov-elliptic1} for further details. \end{rem}

\subsection*{Rank $3$ projective vector bundles on $\mathbb{P}^1$} 
For better readability, we adopt the following notation for the rank $3$ projective vector bundles on $\mathbb{P}^1.$ By the classification of vector bundles on the projective line, every rank $3$ projective vector bundle can be written as
\[  \cO \oplus \cO(d_1) \oplus \cO(d_2) \]
for some integers $d_1 \geq d_2 \geq 0.$ Thus we can assume that all the elements in 
$\PBun_3 \mathbb{P}^1$ are of some of types below: 
\[\varepsilon_0 := \cO \oplus \cO \oplus \cO, \quad 
\varepsilon_{d} :=  \cO \oplus \cO \oplus \cO(d), \quad 
\varepsilon_{d_1,d_2} :=  \cO \oplus \cO(d_1) \oplus \cO(d_2) \]
for some integers $d>0$ and $d_2 \geq d_1 >0.$

\begin{prop} \label{prop-x2} Let $\P^1$ be the projective line defined over the finite field $\Fq$ and $x$ be a degree one closed point at $\P^1$. Then $\overline{\mathscr{G}}_{x,2}^{(3)}$ is given as follows:

\begin{enumerate}[label=\textbf{{\upshape(\roman*)}}, wide=0pt]
\item \label{i1} 
\[ \mathcal{V}_{x,2}(\varepsilon_{0}) = \Big\{ \big(\varepsilon_{0}, \varepsilon_{1}, q^2 + q + 1\big)  \Big\}. \] 

\item \label{i2} Let $d$ be a positive integer, then
\[ \mathcal{V}_{x,2}(\varepsilon_{d}) = \Big\{ \big(\varepsilon_{d}, \varepsilon_{d+1}, q^2\big),
\big(\varepsilon_{d}, \varepsilon_{1,d}, q +1 \big)  \Big\}. \] 

\item \label{i3} Let $d$ be a positive integer, then
\[ \mathcal{V}_{x,2}(\varepsilon_{d,d}) = \Big\{ \big(\varepsilon_{d,d}, \varepsilon_{d,d+1}, q^2+q\big),
\big(\varepsilon_{d,d}, \varepsilon_{d-1,d-1}, 1 \big)  \Big\}. \]

\item \label{i4} Let $d_1,d_2$ be positive integers with $d_1 < d_2$, then
\[ \mathcal{V}_{x,2}(\varepsilon_{d_1,d_2}) = \Big\{ \big(\varepsilon_{d_1,d_2}, \varepsilon_{d_1+1,d_2}, q\big),
\big(\varepsilon_{d_1,d_2}, \varepsilon_{d_1,d_2+1}, q^2 \big),
\big(\varepsilon_{d_1,d_2}, \varepsilon_{d_1-1,d_2-1}, 1 \big)  \Big\}. \]
\end{enumerate}
\end{prop}

\begin{proof}  By \cite[Thm. 2.6]{roberto-graphs} the sum of multiplicities of edges origin in a fixed vertex must sum up $q^2+q+1$. The proposition follows from \cite[Ex. 2.3]{roberto-graphs}.
\end{proof}

\begin{prop} \label{prop-x1} Let $\P^1$ be the projective line defined over the finite field $\Fq$ and $x$ be a degree one closed point at $\P^1$. Then $\overline{\mathscr{G}}_{x,1}^{(3)}$ is given as follows:

\begin{enumerate}[label=\textbf{{\upshape(\roman*)}}, wide=0pt]
\item \label{i1} 
\[ \mathcal{V}_{x,1}(\varepsilon_{0}) = \Big\{ \big(\varepsilon_{0}, \varepsilon_{1,1}, q^2 + q + 1\big)  \Big\}. \] 

\item \label{i2} Let $d$ be a positive integer, then
\[ \mathcal{V}_{x,1}(\varepsilon_{d}) = \Big\{ \big(\varepsilon_{d}, \varepsilon_{1,d+1}, q^2+q\big),
\big(\varepsilon_{d}, \varepsilon_{d-1}, 1 \big)  \Big\}. \] 

\item \label{i3} Let $d$ be a positive integer, then
\[ \mathcal{V}_{x,1}(\varepsilon_{d,d}) = \Big\{ \big(\varepsilon_{d,d}, \varepsilon_{d+1,d+1}, q^2\big),
\big(\varepsilon_{d,d}, \varepsilon_{d-1,d}, q+1 \big)  \Big\}. \]

\item \label{i4} Let $d_1,d_2$ be positive integers with $d_1 < d_2$, then
\[ \mathcal{V}_{x,1}(\varepsilon_{d_1,d_2}) = 
\Big\{ \big(\varepsilon_{d_1,d_2}, \varepsilon_{d_1+1,d_2+1}, q^2\big),
\big(\varepsilon_{d_1,d_2}, \varepsilon_{d_1,d_2-1}, 1 \big),
\big(\varepsilon_{d_1,d_2}, \varepsilon_{d_1-1,d_2}, q \big)  \Big\}. \]
\end{enumerate}
\end{prop}

\begin{proof} This follows from last proposition coupled with \cite[Cor. 2.5]{rov-elliptic1}. 
\end{proof}

Above description of the graphs $\overline{\mathscr{G}}_{x,1}^{(3)}$ and $\overline{\mathscr{G}}_{x,2}^{(3)}$ allows us to state and prove the main theorem of this section.

\begin{thm}\label{thm-eigenforms} We fix $x \in |\P^1|$ of degree one and $\lambda_1, \lambda_2 \in \C$. Let $\overline{\E} \in \PBun_3 \P^1$ and $f \in \mathcal{A}^K(x, \lambda_1, \lambda_2)$.  Then $f(\overline{\E})$ is determined by the values of $\lambda_1, \lambda_2$ and $f(\varepsilon_0)$. In particular,
\[ \dim \mathcal{A}^K(x, \lambda_1, \lambda_2) =1.\] 
\end{thm} 

\begin{proof} Since $f \in \mathcal{A}^K(x, \lambda_1, \lambda_2),$ by definition
$\Phi_{x,1}(f) = \lambda_1 f$ and    $\Phi_{x,2}(f) = \lambda_2 f$. 

Let $d,d_1,d_2$ be positive integers with $d_1 < d_2.$ 
From Proposition \ref{prop-x1} yields
\begin{equation}\tag{1.1} \label{eq1.1}
\lambda_1 f(\varepsilon_0) = (q^2 +q+ 1) f(\varepsilon_{1,1})
\end{equation} 
\begin{equation}\tag{1.2} \label{eq1.2}
\lambda_1 f(\varepsilon_d) = (q^2 +q) f(\varepsilon_{1,d+1}) + f(\varepsilon_{d-1}) 
\end{equation} 
\begin{equation}\tag{1.3} \label{eq1.3}
\lambda_1 f(\varepsilon_{d,d}) = q^2 f(\varepsilon_{d+1,d+1}) + (q+1) f(\varepsilon_{d-1,d})
\end{equation} 
\begin{equation}\tag{1.4} \label{eq1.4}
\lambda_1 f(\varepsilon_{d_1,d_2}) = q^2 f(\varepsilon_{d_1+1,d_2+1}) +  f(\varepsilon_{d_1,d_2-1}) + q f(\varepsilon_{d_1-1,d_2}).
\end{equation} 
From Proposition \ref{prop-x2} yields
\begin{equation}\tag{2.1} \label{eq2.1}
\lambda_2 f(\varepsilon_0) = (q^2 +q+ 1) f(\varepsilon_{1})
\end{equation} 
\begin{equation}\tag{2.2} \label{eq2.2}
\lambda_2 f(\varepsilon_d) = q^2 f(\varepsilon_{d+1}) + (q+1) f(\varepsilon_{1,d}) 
\end{equation} 
\begin{equation}\tag{2.3} \label{eq2.3}
\lambda_2 f(\varepsilon_{d,d}) = (q^2+q) f(\varepsilon_{d,d+1}) + f(\varepsilon_{d-1,d-1})
\end{equation} 
\begin{equation}\tag{2.4} \label{eq2.4}
\lambda_2 f(\varepsilon_{d_1,d_2}) = q f(\varepsilon_{d_1+1,d_2}) + q^2 f(\varepsilon_{d_1,d_2+1}) + f(\varepsilon_{d_1-1,d_2-1}).
\end{equation}
From \eqref{eq1.1} (resp. \eqref{eq2.1}) we can write $f(\varepsilon_{1,1})$ and (resp. $f(\varepsilon_{1})$) in terms of $\lambda_1$ (resp. $\lambda_2$) and $f(\varepsilon_0)$. 
Suppose by induction hypothesis that $f(\overline{\E})$ is determined by $\lambda_1, \lambda_2$ and $f(\varepsilon_0)$ for $\overline{\E}$ equals to $\varepsilon_{d'}, \varepsilon_{d',d'}$ and $\varepsilon_{d_1',d_2'}$ for  all $d' \leq d$, $d_1' \leq d_1$ and $d_2' \leq d_2$ with 
$d_1' \leq d_2'.$ 

The equations $\eqref{eq2.2}$ and $\eqref{eq1.2}$ (in this order) yields
\[f(\varepsilon_{d+1}) = \frac{1}{q^2}\left( \lambda_2 f(\varepsilon_d) 
- \tfrac{q+1}{q^2+q} \Big(\lambda_1 f(\varepsilon_d) - f(\varepsilon_{d-1})\Big) \right)\]
and thus by induction hypothesis $f(\varepsilon_d)$ is determined by $\lambda_1, \lambda_2$ and $f(\varepsilon_0)$ for all positive integer $d$. 

The equations $\eqref{eq1.3}$ and $\eqref{eq2.3}$ yields
\[f(\varepsilon_{d+1,d+1}) = \frac{1}{q^2}\left( \lambda_1 f(\varepsilon_{d,d}) 
- \tfrac{q+1}{q^2+q} \Big(\lambda_2 f(\varepsilon_{d-1,d-1}) - f(\varepsilon_{d-2,d-2})\Big) \right)\]
and thus by induction hypothesis $f(\varepsilon_{d,d})$ is determined by $\lambda_1, \lambda_2$ and $f(\varepsilon_0)$ for all positive integer $d$.

The equation \eqref{eq2.4} yields
\[f(\varepsilon_{d_1,d_2+1}) = \frac{1}{q^2}\left( \lambda_2 f(\varepsilon_{d_1,d_2}) 
- q f(\varepsilon_{d_1+1,d_2}) - f(\varepsilon_{d_1-1,d_2-1}) \right)\]
and by induction hypothesis we are left to show that $f(\varepsilon_{d_1+1,d_2})$
is determined by $\lambda_1, \lambda_2$ and $f(\varepsilon_0).$ If $d_1 +1 = d_2$, then we are done by the case $f(\varepsilon_{d,d})$ above. Otherwise, if $d_2 > d_1+1$, we apply \eqref{eq1.4} replacing $d_2$ by $d_2-1$ and obtain that 
\[f(\varepsilon_{d_1+1,d_2}) = \frac{1}{q^2}\left( 
\lambda_1 f(\varepsilon_{d_1,d_2-1}) -  f(\varepsilon_{d_1,d_2-2}) 
- q f(\varepsilon_{d_1-1,d_2-1}) \right).\]
Thus both $f(\varepsilon_{d_1,d_2+1})$ and $f(\varepsilon_{d_1+1,d_2})$ are determined by 
$\lambda_1, \lambda_2$ and $f(\varepsilon_0)$. 
Finally, identity \eqref{eq1.4} yields that 
\[f(\varepsilon_{d_1+1,d_2+1}) = \frac{1}{q^2} \left( \lambda_1 f(\varepsilon_{d_1,d_2}) -  f(\varepsilon_{d_1,d_2-1}) - q f(\varepsilon_{d_1-1,d_2})
\right) \]
i.e.\ $f(\varepsilon_{d_1+1,d_2+1})$ is determined by $\lambda_1, \lambda_2$ and $f(\varepsilon_0)$. By induction hypothesis we conclude that $f(\varepsilon_{d_1,d_2})$ is determined by $\lambda_1, \lambda_2$ and $f(\varepsilon_0)$ for all positive integers $d_1,d_2$ with $d_1< d_2$.  
\end{proof}

As a byproduct of previous theorem we have Theorem \ref{thm-cusp} for $n=3$.

\begin{cor} There are no unramified cusp forms for $\PGL_3$ over the projective line. 
\end{cor}

\begin{proof} Let $T$ be diagonal torus of $\GL_3$. Given $\lambda_1, \lambda_2 \in \C$ and $x \in \P^1$ a closed point of degree one, there exists a nontrivial Eisenstein series induced from an unramified character of $T$ which is an eigenfunction for $\Phi_{x,r}$ ($r=1,2$) with eigenvalues $\lambda_1, \lambda_2$, see \cite[Thm. 1.7.7]{valdir-thesis}. 
It follows from \cite[Thm. 1.7.9]{valdir-thesis} and above theorem that 
\begin{equation} \label{eq-tr}
A_{0}^K \cap   \mathcal{A}^K(x, \lambda_1, \lambda_2) = \{0\}.
\end{equation}
Furthermore, $\mathcal{A}_{0}$ splits as a direct sum of irreducible representations. Therefore, 
we can write every $f \in \mathcal{A}_{0}^{K}$ as a sum of eigenforms and thus $f=0$ by above (\ref{eq-tr}).

\end{proof}

   
\section{Toroidal autormophic forms}   \label{sec-toroidal}

We define the space of toroidal automorphic forms for any global function field $F$ and, using the results from previous sections, we derive some results when $F$ is the function field of $\mathbb{P}^1$. 

Let $F$ be any global function field over $\Fq$ and $E/F$ be a separable field extension of degree $n$.
 Choosing a basis for $E$ over $F$ gives an embedding of $E^{*}$ in $G(F)$ and a non-split maximal torus $T \subseteq G$ with $T(F)= E^{*}$ and $T(\mathbb{A}_F) = \mathbb{A}_{E}^{*}$. In this case we say that $T$ is associates to $E/F$. We refer \cite[Def. 1.5.1]{oliver-thesis} for the definitions of non-split and maximal torus. 

\begin{df} Let $T$ be a maximal torus of $\GL_n$ over $F$ associated with a separable extension $E/F$ of degree $n$. Let $\mathbb{A}:= \mathbb{A}_F$. Endow $T(\mathbb{A})$ and $T(F)Z(\mathbb{A}) $ with the Haar measures  and $T(F)Z(\mathbb{A})\setminus T(\mathbb{A})$ with the quotient measure. For $f \in \mathcal{A}$ we define 
\[f_T(g) := \int_{T(F)Z(\mathbb{A})\setminus T(\mathbb{A})} f(tg)dt\]
the \textit{toroidal integral} of $f$ along $T$.
\end{df}

\begin{rem} The quotient $T(F)Z(\mathbb{A})\setminus T(\mathbb{A}) $ is compact, see \cite[Pag. 42]{valdir-thesis}. \end{rem}

\begin{df} \label{def-toroidal} Let $T$ be a maximal torus of $\GL_n$ over $F$ associated with a separable extension $E/F$ of degree $n$. We define
\[\mathcal{A}_{tor}(E) := \big\{f \in \mathcal{A} \;|\; f_T(g)=0, \forall g \in G(\mathbb{A})\big\}\]
the space of $E$\textit{-toroidal automorphic forms}, and
\[\mathcal{A}_{tor} = \bigcap_{E/F} \mathcal{A}_{tor}(E)\]
the space of \textit{toroidal automorphic forms}, where $E/F$ runs over the separable extensions of degree $n$. 
\end{df}

\begin{rem} The spaces $\mathcal{A}_{tor}(E)$ do not depend on the choice of the basis for $E$ over $F$, see \cite[Rem. 2.1.3]{valdir-thesis}. \end{rem}

\begin{thm}\label{thm-5.5} Let $F $ be the function field of $\mathbb{P}^1$ defined over $\Fq$ and $E$ be the constant field extension of $F$ of degree $3$, i.e.\ $E=\mathbb{F}_{q^3}F.$ If $x$ a degree one place of $F$ and $\lambda_1, \lambda_2 \in \C$, then
\[\mathcal{A}_{tor}(E) \cap \mathcal{A}^K(x, \lambda_1, \lambda_2) = \{0\}\]
there does not exist nontrivial $\Phi_{x,r}$-toroidal eigenforms for $n=3$ and $r=1,2$. 
\end{thm}

\begin{proof}Let $T$ be the $3$-dimensional torus associated to $E/F$, where $E = \mathbb{F}_{q^3}F.$ Thus $E$ is the function field of $\mathbb{P}_{3}^{1} := \mathbb{P}^1 \otimes_{\mathrm{Spec}\Fq} \mathrm{Spec} \mathbb{F}_{q^3}$, the extension of scalars of $\mathbb{P}^1$. The extension of scalars yields 
\[p: \mathbb{P}_{3}^{1} \rightarrow \P^1\]
the projection map. 
Moreover $p$ induces the inverse image $p^{*} : \Bun_3 X \rightarrow \Bun_3 \mathbb{P}_{3}^{1}$ and the direct image (or trace) $p_{*}: \Bun_1 \mathbb{P}_{3}^{1} \rightarrow \Bun_3 \mathbb{P}^{1}$. 

From \cite[Thm. 2.8.1]{valdir-thesis} 
\[ \int_{T(F)Z(\mathbb{A})\setminus T(\mathbb{A})} f(tg)\mu(t) = c_T \cdot \sum_{[\Line] \in \mathrm{Pic} \mathbb{P}_{3}^{1}/p^{*}\mathrm{Pic}\mathbb{P}^{1}} f(\overline{p_{*}\Line})\]
where   and 
$$c_T = \frac{vol(T(F)Z(\A)\setminus T(\A))}{\# (\Pic\mathbb{P}_{3}^{1}/p^*\Pic \P^1)}$$

Hence $f \in \mathcal{A}(x, \lambda_1, \lambda_2)$ is $E$-toroidal, thus  
$f(\varepsilon_0) =0$. 
Therefore Theorem \ref{thm-eigenforms} implies that if $f(\varepsilon_0)=0$, then $f$ is trivial.
\end{proof}

Since the zeta function of $\P^1$ has no zeros, a possible connection with the space of toroidal automorphic forms lead us to the following conjecture. 

\begin{conj} For all $n \geq 0$, $\mathcal{A}_{tor} = \{0\}.$
\end{conj}

We end this article with the following partial answer for previous conjecture. 

\begin{thm} Let $F $ be the function field of $\mathbb{P}^1$ defined over $\Fq$ and $E$ be the constant field extension of $F$ of degree $3$, i.e.\ $E=\mathbb{F}_{q^3}F.$ Then,
$\mathcal{A}_{tor}(E) \cap \mathcal{A}^K = \{0\} $ and therefore 
$\mathcal{A}_{tor} \cap \mathcal{A}^K = \{0\}$ for $n=3$.
\end{thm}

\begin{proof} We suppose by contradiction that $\mathcal{A}_{tor}(E) \cap \mathcal{A}^K \neq \{0\}.$ Hence, let $f \in \mathcal{A}_{tor}(E) \cap \mathcal{A}^K$ such that $f \neq 0,$ By admissibility condition, $V = \mathcal{H}_K \cdot f$ is a finite dimensional vector space, Moreover, 
 $V$ is invariant by the action of $\Phi_{x,1}, \Phi_{x,2} \in \mathcal{H_K}$ for some $x \in |\mathbb{P}^1|$ of degree one. Thus, there exists a $\Phi_{x,r}$-eigenform (for $r=1,2$) in $V$, which disagree with Theorem \ref{thm-5.5}. 
\end{proof}

\newcommand{\etalchar}[1]{$^{#1}$}


\end{document}